\pdfoutput=1
\documentclass{amsart}
\usepackage{amsmath}
\usepackage{amsthm}
\usepackage{amssymb}
\usepackage[german,english]{babel}
\usepackage[utf8]{inputenc}
\usepackage{microtype}
\usepackage[T1]{fontenc}
\usepackage{lmodern}
\usepackage{epsfig,fancybox,color}
\usepackage{enumerate}
\usepackage{setspace}
\usepackage{mathtools}  
\usepackage[linesnumbered,ruled]{algorithm2e}
\usepackage{todonotes}
\usepackage{tikz}
\usepackage{tikz-cd} 
\usetikzlibrary{arrows.meta}
\usepackage{hyperref}
\definecolor{darkred}{RGB}{160,0,0}
\definecolor{darkblue}{RGB}{0,0,160}
\hypersetup{
  colorlinks,
  citecolor=darkblue,
  filecolor=black,
  linkcolor=darkblue,
  urlcolor=darkblue
}

\allowdisplaybreaks

\usepackage[
  text={410pt,575pt},
  headheight=9pt,
  centering
]{geometry}

\renewcommand{\>}{\rangle}
\newcommand{\<}{\langle}

\newcommand\defas{\coloneqq}

\newcommand{\gin}{\operatorname{gin}}
\newcommand{\Pf}{\operatorname{Pf}}

\usepackage{
     amscd,
     thmtools}

\theoremstyle{plain}
\newtheorem*{thm*}{Theorem}
\newtheorem*{lm*}{Lemma}
\newtheorem*{ex*}{Example}
\newtheorem{thm}{Theorem}[section]
\newtheorem{lm}[thm]{Lemma}
\newtheorem{cor}[thm]{Corollary}
\newtheorem{prop}[thm]{Proposition}

\theoremstyle{definition}

\declaretheorem[sibling=thm,name=Remark,qed={$\diamondsuit$}]{re}

\newcommand{\QQ}{{\mathbb Q}}
\newcommand{\ZZ}{{\mathbb Z}}

\newcommand{\PP}{{\mathbb P}}

	{\begin{figure} \begin{center}}%
	{\end{center} \end{figure}}

\newcommand{\cha}{\operatorname{char}}

\newcommand{\id}{\operatorname{id}}

\newcommand{\Gr}{\operatorname{Gr}\nolimits}

\newcommand{\GL}{\operatorname{GL}\nolimits}

\renewcommand{\phi}{\varphi}

\begin{document}

\title{No short polynomials vanish on bounded rank matrices}
\author{Jan Draisma}
\author{Thomas Kahle}
\author{Finn Wiersig}
\date{\today}


\thanks{JD was partly supported by Vici Grant 639.033.514 from the
  Netherlands Organisation for Scientific Research (NWO) and by
  project grant 200021\textunderscore 191981 from the Swiss National
  Science Foundation (SNSF).  TK and FW were supported by the German
  Research Foundation DFG -- 314838170, GRK~2297 MathCoRe}

\makeatletter
\@namedef{subjclassname@2020}{\textup{2020} Mathematics Subject Classification}
\makeatother


\begin{abstract}
  We show that the shortest nonzero polynomials vanishing on
  bounded-rank matrices and skew-symmetric matrices are the
  determinants and Pfaffians characterising the rank.  Algebraically,
  this means that in the ideal generated by all $t$-minors or
  $t$-Pfaffians of a generic matrix or skew-symmetric matrix one
  cannot find any polynomial with fewer terms than those determinants
  or Pfaffians, respectively, and that those determinants and
  Pfaffians are essentially the only polynomials in the ideal with
  that many terms.  As a key tool of independent interest, we show
  that the ideal of a very general $t$-dimensional subspace of
  an affine $n$-space does not contain polynomials with fewer than
  $t+1$ terms.
\end{abstract}

\maketitle

\section{Introduction}
In many areas of computational mathematics, sparsity is an essential
feature used for complexity reduction.  Sparse mathematical objects
often allow more compact data structures and more efficient
algorithms.  We are interested in sparsity as a complexity measure for
polynomials, where, working in the monomial basis, it means having few
terms.  This augments the usual degree-based complexity measures such
as the Castelnuovo--Mumford regularity.

Sparsity based complexity applies to geometric objects, as
well.  If
$X\subset K^{n}$ is a subset of affine $K$-space, one can ask for the
shortest polynomial that vanishes on~$X$.  A~{\em monomial} vanishes on $X$
if and only if $X$ is contained in the union of the coordinate
hyperplanes.  That $X$ is cut out by {\em binomials} can be characterised
geometrically using the log-linear geometry of binomial
varieties~\cite[Theorem~4.1]{eisenbud96:_binom_ideal}.  Algorithmic
tests for single binomials vanishing on $X$ are available both
symbolically~\cite{JKK17} and
numerically~\cite{hauenstein2021binomiality}.  We ask for the shortest
polynomial vanishing on~$X$, or algebraically, the shortest polynomial
in an ideal of the polynomial ring.  The shortest polynomials
contained in (principal) ideals of a univariate polynomial ring have
been considered in \cite{giesbrecht2010computing}.  Computing the
shortest polynomials of an ideal in a polynomial ring seems to be a
hard problem with an arithmetic flavor.  Consider Example~2 from
\cite{JKK17}: For any positive integer $n$, let
$I_n=((x-z)^2,nx-y-(n-1)z)\subseteq\QQ[x,y,z]$. The ideals $I_n$ all
have Castelnouvo-Mumford regularity 2 and are primary over
$(x-z,y-z)$, so in a sense they are all very similar.
However, $I_{n}$ contains the binomial $x^n-yz^{n-1}$ and
there is no binomial of degree less than $n$ in~$I_{n}$.  This means
that the syzygies and also the primary decomposition carry no information
about short polynomials.  It is unknown to the authors if a Turing machine can
decide if an ideal contains a polynomial with at most $t$ terms.

In this paper we show that determinants are the shortest nonzero
polynomials that vanish on the set of fixed-rank matrices and that,
moreover, they are essentially the only shortest polynomials in the
determinantal ideal (Theorem~\ref{thm:Mats}).  A~variant of the proof
yields a similar result (Theorem~\ref{thm:skew}) for skew-symmetric
matrices, where Pfaffians, the square roots of determinants, are the
shortest vanishing nonzero polynomials.  Their number of terms is the
double factorial $(r+1)!! \defas (r+1)(r-1)\cdots$.  Both proofs rely
on Proposition~\ref{prop:Linear}, a bound for the number of terms of
polynomials vanishing on very general linear spaces.  In
Section~\ref{sec:symmetric-matrices} we briefly discuss the case of
bounded rank symmetric matrices, which however remains mostly open!

Our proofs have geometric aspects, and for these it is convenient to
work with algebraically closed fields. However, Theorems~\ref{thm:Mats}
and~\ref{thm:skew} immediately imply that the corresponding ideals over
arbitrary fields contain no shorter polynomials than determinants and
Pfaffians, respectively; see Corollaries~\ref{cor:Mats}
and~\ref{cor:skew}. In the determinant case, this improves a lower
bound of $(r+1)!/2$ terms established by the last two authors via purely
algebraic methods \cite{kahle2021short}.

\subsection*{Notation and conventions}

In everything that follows there are fixed bases with respect to which
any sparsity is considered.  We use the standard basis of $K^{n}$ and
the monomial basis for polynomials. We write $K[x_1,\ldots,x_n]_d$ for
the space of homogeneous polynomials of degree $d$ in the variables
$x_1,\ldots,x_n$ with coefficients from the field $K$. Except in
Corollaries~\ref{cor:Mats} and~\ref{cor:skew}, we assume that $K$ is
algebraically closed. The characteristic of $K$ is arbitrary.

\subsection*{Acknowledgements}

The authors thank Rob Eggermont for useful conversations and his proof
of Lemma~\ref{lm:S3}.

\section{No short polynomials vanish on very general subspaces}
\label{sec:genSubspace}
If $X$ is an irreducible algebraic variety over $K$, we say
that a {\em sufficiently general} $x \in X$ has a certain property if there
exists a Zariski open and dense subset $Y\subset X$ such
that all $x \in Y$ have that property. The open and dense
subset $Y$ is typically not made explicit, and may moreover
shrink finitely many times in the course of a proof as further assumptions are
imposed on $x$. This notion of genericity is common in
algebraic geometry.  

Another common notion from algebraic geometry that we will need is the
following. We say that a {\em very general} $x \in X$ has a certain
property if there is a countable collection of proper, Zariski-closed
subsets of $X$, defined over $K$, such that any $x$ outside their
union satisfies the property.  If the ground field $K$ is too small,
then such very general $x$ may exist only over a field extension of
$K$.  This is no problem in our application to varieties of
bounded-rank matrices and skew-symmetric matrices, where, to prove our
results, we may always extend the field as desired.  However, in our
result on linear spaces, we will require that the space be very
general.

Indeed, we consider properties of a sufficiently or
even very general
$r$-dimensional linear subspace $U\subset K^{n}$.  In this case, $X$
is understood to be the Grassmannian $\Gr_r(K^n)$, and $U$ is called
sufficiently general if the point in $\Gr_r(K^n)$
representing it is sufficiently general. 
For example, when $U\in \Gr_r(K^n)$ is sufficiently general, any $r$
coordinates are linearly independent on~$U$, and hence the shortest
linear polynomials vanishing on $U$ have $r+1$ terms.  For instance,
$c_1 x_1 + \cdots + c_{r+1} x_{r+1}=0$ holds on $U$ for certain
nonzero $c_1,\ldots,c_{r+1} \in K$.  Multiplying such $(r+1)$-term
linear polynomials by monomials, or, if $K$ has positive
characteristic $p>0$, raising them to $p^{e}$-th powers yields short
polynomials of higher degree also vanishing on~$U$.
A key step in our argument is to show that these are all shortest
polynomials vanishing on~$U$, at least for very general $U$. 

To formulate and prove our results in a characteristic independent
manner, let $p$ be the characteristic exponent of $K$, i.e., $p:=1$ if
$\cha K=0$ and $p:=\cha K$ otherwise.

\begin{prop}
  \label{prop:Linear}
  Let $n\geq r$ and $d$ be nonnegative integers and let $U$ be a
  very general $r$-dimensional subspace of~$K^n$. Then a
  nonzero polynomial $f \in K[x_1,\ldots,x_n]$ that vanishes
  identically on $U$ has at least $r+1$ terms. If $r \neq 1$, then
  equality holds if and only if $f$ has the form
  $u \cdot ((c_{1} x_{i_1})^{p^e} + \cdots + (c_{r+1}
  x_{i_{r+1}})^{p^e})$ for some monomial $u$, distinct indices
  $i_1<\ldots<i_{r+1}$, nonnegative integer $e$, and
  $\sum_j c_j x_{i_j}$ a linear form that vanishes on~$U$.
\end{prop}

\begin{re} \label{re:Dim1} If $p=1$ or $e=0$, the second factor is
  just a linear form.  Furthermore, the requirement that $r \neq 1$ is
  necessary for the characterisation of the shortest polynomials.
  Indeed, if $r=1$, then some linear form $c_1 x_1 + c_2 x_2$ vanishes
  on $U$, and then so does the binomial $c_1^2 x_1^2 - c_2^2 x_2^2$,
  which is not of the shape in the proposition.  If $r=1$, then the
  1-dimensional torus $K^*$ acts, via scaling, on $U$ with a dense
  orbit, and thus the ideal of $U$ is a binomial ideal.  Binomial
  ideals are linearly spanned by the binomials they contain, which
  shows that they contain many binomials.
\end{re}

\begin{re}
  We do not know whether {\em very general} in
  Proposition~\ref{prop:Linear} can be replaced by {\em
  sufficiently general}. In our proof below, we require that
  $U$ avoids countably many Zariski-closed subsets of the
  Grassmannian. 
\end{re}

Our proof of the lower bound $r+1$ in Proposition~\ref{prop:Linear} is
quite concise; readers only interested in this can skip directly to
the proof.  The proof of the characterisation of equality, however, is
more involved.  It requires the following application of Gr\"obner
bases.  The {\em reverse lexicographic order} on the space
$K[x_1,\ldots,x_n]_d$ is defined by $x^\alpha > x^\beta$ if for the
{\em largest} $j$ with $\alpha_j \neq \beta_j$ we have
$\alpha_j<\beta_j$. Thus, the monomial basis of this space, in
decreasing order, is
\[ x_1^d, x_1^{d-1} x_2,\ldots, x_2^d,
x_1^{d-1}x_3,x_1^{d-2}x_2x_3,\ldots,x_2^{d-1}x_3,x_1^{d-2}x_3^2,\ldots,x_n^d.
\]
The {\em generic initial space} $\gin(V)$ of a subspace
$V \subseteq K[x_1,\ldots,x_n]_d$ is the space spanned by the leading
monomials of elements of $g V$, which for a sufficiently general
element of $g\in \GL_n$ does not depend on~$g$.  This space has two
important properties. First, it is in the closure of the $\GL_n$-orbit
of $V$ in $\Gr_r(K^n)$, and second, it is stable under the Borel
subgroup of $\GL_n$ that stabilises the chain of subspaces
\[\< x_1 \> \supset \< x_1,x_2 \> \supset \ldots \supset \<
  x_1,\ldots,x_n \>. \]
We employ the following result from Gr\"obner basis theory.

\begin{lm} \label{lm:Gin} Let $d \in \ZZ_{\geq 1}$.  Suppose that a
  linear space $V \subseteq K[x_1,\ldots,x_n]_d$ has
  $\gin(V)=x_1^{d-p^e} \cdot \< x_1^{p^e},\ldots,x_{s}^{p^e} \>$ for
  some $s$ with $3 \leq s \leq n$ and some $e \in \ZZ_{\geq 0}$. Then
  $V=f \cdot \< \ell_1^{p^e},\ldots,\ell_s^{p^e} \>$ for some
  $f \in K[x_1,\ldots,x_n]_{d-p^e}$ and linear forms
  $\ell_1,\ldots,\ell_s \in K[x_1,\ldots,x_n]_1$.
\end{lm}
In characteristic zero, this is a special case of \cite[Main
Theorem]{Floystad99}.  Our proof follows the strategy of the proof
there, but replaces algebraic arguments involving differentiation by
geometric arguments that suffice in our setting.

\begin{proof}
  The proof can be split as follows.  If $d=p^e$, then $V$ consists of
  $d$-th powers of linear forms; while if $d>p^e$, then it suffices to
  show that $V=\tilde{f} \cdot \tilde{V}$ for some homogeneous
  $\tilde{f}$ of positive degree $d-\tilde{d}>0$.  In this case
  $\gin(\tilde{V})=x_1^{\tilde{d}-p^e} \cdot \<
  x_1^{p^e},\ldots,x_s^{p^e} \>$ and the argument applies
  to~$\tilde{V}$.  If $d=1$, then the first statement obviously holds,
  so we may assume that $d>1$.

  We prove both statements first for $s=n$. For a sufficiently
  general $g \in \GL_n$, the space $g V$ contains a polynomial $f$
  with leading monomial $x_1^{d-p^e} x_n^{p^e}$.  By definition of the
  reverse lexicographic order, $f$ is divisible
  by~$x_n^{p^e}$. Consequently, $V$ itself contains a nonzero
  polynomial divisible by $g^{-1} x_n^{p^e}$, namely, $g^{-1}
  f$. Since this holds for any sufficiently general $g$, $V$ contains
  a nonzero multiple of the $p^e$-th power of any sufficiently general
  linear form.

Let $L:=K[x_1,\ldots,x_n]_1$ be the space of linear forms. 
Consider the incidence variety 
\[ Z:=\{([\ell],[f]) \in \PP L \times \PP V \mid \ell^{p^e} \text{ divides }
f \}. \]
By the previous paragraph, the projection $Z \to \PP L$ is dominant. Since $Z$ is
projective, it is in fact surjective. Replace $Z$ by an
irreducible component that maps surjectively to $\PP L$, so that $\dim(Z)
\geq \dim \PP L=n-1=\dim \PP V$. Since the fibres of $Z \to \PP V$ are
finite---each nonzero element of $V$ is divisible only by finitely many
$p$-th powers of linear forms---we find that also $\dim(Z) \leq \dim \PP
V$. Hence $\dim(Z)=\dim \PP V=\dim \PP L$. 

This implies two things: First, any sufficiently general fibre of $Z \to
\PP L$ has dimension zero---and since these fibres are projective linear
spaces, a sufficiently general fibre is a single point. And second,
$Z \to \PP V$ is surjective, so any element of $V$ is divisible by
some $p^e$-th power of a linear form. If $d=p^e$, then we are done,
so we may henceforth assume that $d>p^e$.

Now fix a basis $f_1,\ldots,f_n$ of $V$ and let $X \subseteq
\PP^{n-1}$ be the affine open subset where $f_1 \neq 0$. Consider the morphism
\[ 
\phi\colon X \to K^{n-1},\ p \mapsto
\left(\frac{f_2(p)}{f_1(p)},\ldots,\frac{f_n(p)}{f_1(p)} \right).
\]
Since doing this for $\gin(V)$ would yield an image closure of
dimension $n-1 \geq 3-1=2$ (take the $f_i$ equal to
$x_1^{d-p^e} x_i^{p^e}$ to see this) and the image closure dimension
can only go down in a limit, $\overline{\phi(X)}$ has dimension at
least $2$.  Hence, by Bertini's theorem \cite[Theor\`eme
6.3(4)]{Jouanolou83}, the preimage in $X$ of a sufficiently general
affine hyperplane $H$ in $K^{n-1}$ is irreducible.  If $H$ has the
equation $a_1y_{1} + a_2 y_2 + \cdots + a_n y_n =0$, then $\phi^{-1}(H)$
has the equation $h:=a_1 f_1 + a_2 f_2 \cdots + a_n f_n=0$. The
left-hand side is an element of $V$ and hence factors as
$h=k \cdot \ell^{p^e}$ with $k$ of degree $d-p^e>0$ and $\ell$ a
linear form.  Irreducibility of $\phi^{-1}(H)$ implies that either the
hypersurface in $\PP^{n-1}$ defined by $k$ is disjoint from the open
set $X$, or the hyperplane defined by $\ell$ is. In the latter case,
$\ell$ is a divisor of $f_1$, but this can be avoided by generality
of~$H$.  Hence $k$ is a divisor of some power of~$f_1$.  In total
$f_1$ has a nontrivial gcd with every element in $V$ and thus all
elements in $V$ have a nontrivial gcd, as desired. This concludes the
proof for the case where $s=n$.

Now assume $n>s \geq 3$. For any sufficiently general $g \in \GL_n$,
let $\tilde{V} \subseteq K[x_1,\ldots,x_s]$ be the space obtained from
$gV$ by setting the variables $x_{s+1},\ldots,x_n$ to zero. Then
$\gin(\tilde{V})=x_1^{d-p^e} \cdot \< x_1^{p^e},\ldots,x_s^{p^e} \>$
and hence, by the above,
$\tilde{V}=\tilde{f} \cdot \< x_1^{p^e},\ldots,x_s^{p^e} \>$ for a
nonzero homogeneous polynomial $\tilde{f}$.
We again distinguish two cases. If $d=p^e$ and some nonzero polynomial
$f$ in $V$ is not a linear combination of $d$-th powers of variables,
then $f$ is not an additive polynomial, and hence not additive on
sufficiently general $s$-dimensional subspaces of~$K^n$ (here we only
need that $s \geq 2$). This implies that $gf$ with the last $n-s$
variables set to zero is not a linear combination of $d$-th powers of
variables, contradicting the previous paragraph.

Now assume that $d>p^e$ and let $Y \subseteq \PP^{n-1}$ be the variety
defined by the polynomials in $V$. Then the penultimate
paragraph implies that the intersection of 
$Y$ with a sufficiently general codimension-$(n-s)$ subspace
contains a
hypersurface in $\PP^{s-1}$ (defined by $g^{-1}\tilde{f}$,
where $g \in \GL_n$ maps the linear equations for the
subspace to $x_{s+1},\ldots,x_{n}$).
But then $Y$ must
itself have a component of dimension $n-2$, i.e., a hypersurface. This
shows that the elements in $V$ have a nontrivial gcd, and we are done.
\end{proof}

\begin{proof}[Proof of Proposition~\ref{prop:Linear}]
  Let $U \subseteq K^n$ be a very general $r$-dimensional
  subspace with $r \geq 2$. We want to show that polynomials vanishing
  on $U$ have at least $r+1$ terms, and characterise those where
  equality holds. The requirement that $U$ be {\em very}
  general comes from the fact that we have to exclude
  equations for $U$ with fewer than $r+1$ terms of {\em
  varying degrees}. In each fixed degree, {\em sufficiently}
  general suffices. 
  \medskip

  \noindent {\bf Part 1: proof of the lower bound $r+1$.}
  If some polynomial $f$ vanishes on $U$, then every homogeneous
  component of $f$ vanishes on $U$, so we may assume that $f$ is
  homogeneous of some degree $d$.  Consider a space $V$ spanned by $N$
  distinct degree-$d$ monomials $x^{\alpha_i}, i=1,\ldots,N$ in $n$
  variables. The set of $U \in \Gr_r(K^n)$ for which there exists a
  point $[f_1:\ldots:f_N] \in \PP(K^N)$ with $\sum_i f_i x^{\alpha_i}$
  identically zero on $U$ is a closed subset of the Grassmannian
  $\Gr_r(K^n)$.  Since, for a fixed $d$, there are only finitely many
  subsets of the set of degree-$d$ monomials, we may assume that $U$
  lies outside all of these closed subsets that are not the entire
  Grassmannian.  It follows, then, that if such a point
  $[f_1:\ldots:f_N]$ {\em does} exist for $U$, then such a point
  exists for {\em every} $r$-dimensional subspace of $K^n$. We assume
  that this is the case and bound $N$ from below.

  Write $F:=K[x_1,\ldots,x_n]_d$ and consider the incidence variety
  \[ 
    Z:=\{(W,[f],U) \in \Gr_N(F) \times \PP(F) \times \Gr_r(K^n) 
    \mid f \in W \text{ and } f|_U \equiv 0\}
  \]
  and the projection $\pi\colon Z \to \Gr_N(F) \times \PP(F)$. The
  set of points $z \in Z$ for which $\pi^{-1}(\pi(z))$ is
  all of $\{\pi(z)\} \times Gr_r(K^n)$ is a closed subset $Y$ of $Z$, and the image
  of $Y$ in $\Gr_N(F)$ is a closed subset $C$ of $\Gr_N(F)$.  The span $V$
  of the $x^{\alpha_i}$ is a point in $C$, so $C$ is nonempty.  In fact,
  in what follows we replace $C$ by the $\GL_n$-orbit closure of $V$
  in $\Gr_N(F)$.

  By construction, $C$ is a $\GL_n$-stable closed subset of the
  projective variety $\Gr_N(F)$, and hence by Borel's fixed point
  theorem \cite[Theorem~10.4]{BorelLAG}, $C$ contains a point $W$ that
  is stable under the Borel subgroup $B \subseteq \GL_n$ that
  stabilises the flag
  \[ \< x_1 \> \supset \< x_1,x_2 \> \supset \ldots \supset \<
    x_1,\ldots,x_n \>. \]
  Since $B$ contains the torus $(K^*)^n$, $W$ is spanned by
  monomials. Furthermore, these monomials satisfy the following
  well-known property: if $x^\beta \in W$ and $\beta_j>0$, then
  writing $\beta_j=p^e m$ with $p\!\!\not|\, m$ we have
  $x^\alpha:=x^{\beta-p^e e_j + p^e e_i} \in W$ for all $i<j$. Indeed,
  this follows easily by considering the one-parameter subgroup in $B$
  that maps $x_j$ to $x_j + t x_i$ and fixes all other basis elements:
  this maps $x^\beta$ to $x^\beta + m \cdot t^{p^e} \cdot x^\alpha +$
  terms of higher degree in $t$. Since $m \neq 0$ in $K$,
  $B$-invariance of $W$ implies that $x^\alpha \in W$.

  By construction, on every $r$-dimensional subspace of
  $K^n$ some nonzero element of $W$ vanishes identically. 
  Since on the space $K^r \times \{0\}^{n-r}$ no nonzero polynomial in
  the first $r$ variables vanishes, $W$ contains a monomial $x^\beta$
  with $\beta_s>0$ for some $s>r$. Writing $\beta_s=p^e m$ as above,
  we find that $W$ also contains the $s-1$ monomials $x^{\beta-p^e e_s +
  p^e e_i}$ with $i=1,\ldots,s-1$. Hence $W$ has dimension at least $s
\geq r+1$, as desired.
\medskip

\noindent
  {\bf Part 2: Proof of the characterisation.}
  If $\dim W=r+1$ holds, then the previous paragraph shows
  that $s=r+1$, and that we have already listed all
  monomials in $W$. By similar arguments we find that 
  $W=x_1^{d-p^e} \cdot \< x_1^{p^e},\ldots,x_s^{p^e} \>$.
  We have thus established that every $B$-stable element in the
  $\GL_n$-orbit closure of our original space $V$ is this particular
  space $W$. This applies, in particular, to $W=\gin(V)$. But then,
  by Lemma~\ref{lm:Gin}, $V=f \cdot \< \ell_1^{p^e}, \ldots,
  \ell_s^{p^e} \>$ for some polynomial $f$ of degree $d-p^e$ and some
  linear forms $\ell_1,\ldots,\ell_s$. Finally, since $V$ is spanned by
  monomials, $f$ is a monomial and the $\ell_i$ can be taken
  to be variables. This proves the proposition. 
\end{proof}

\section{No short polynomials vanish on bounded-rank matrices}
\label{sec:generic-matrix}
Using Proposition~\ref{prop:Linear} inductively, we can characterise
the shortest polynomials vanishing on fixed rank matrices.

\begin{thm} \label{thm:Mats} Let $m,n,r$ be natural numbers with
  $m,n \geq r$. Then there exists no nonzero polynomial with fewer
  than $(r+1)!$ terms that vanishes on all rank-$r$ matrices in
  $K^{m \times n}$. Moreover, if $r \geq 2$, then every polynomial
  with exactly $(r+1)!$ terms that vanishes on all such matrices is a
  term times the $p^e$-th power of some $(r+1)$-minor, for some
  nonnegative integer~$e$.
\end{thm}

\begin{re} \label{re:Binomials} As in Proposition~\ref{prop:Linear},
  and for the same reason, the case $r=1$ needs to be excluded in the
  second part of the theorem. Indeed, the variety of rank-$1$ matrices
  has a dense $(K^*)^m \times (K^*)^n$-orbit, and hence its ideal is
  spanned by binomials. Most of these binomials are not of the form in
  the theorem. However, we know exactly what they are, namely, (scalar
  multiples of) $x^\alpha - x^\beta$ where the $m \times n$-exponent
  matrices $\alpha$ and $\beta$ satisfy
  $\sum_j \alpha_{ij}=\sum_j \beta_{ij}$ for all $i$ and
  $\sum_i \alpha_{ij}=\sum_i \beta_{ij}$ for all $j$, and
  where $x^\alpha$ is short-hand for $\prod_{i,j}
  x_{ij}^{\alpha_{ij}}$. The proof of
  Theorem~\ref{thm:Mats} proceeds by induction on $r$, and for the
  second part we start with $r=2$, where this characterisation of
  binomials vanishing on rank-one matrices is used.
\end{re}

Before proceeding with the proof, we record a corollary over
arbitrary fields.

\begin{cor} \label{cor:Mats} Let $m,n,r$ be as in
  Theorem~\ref{thm:Mats}, and let $L$ be an arbitrary field. Then the
  ideal $I \subseteq L[x_{ij} \mid (i,j) \in [m] \times [n]]$
  generated by the $(r+1)$-minors of the matrix $x = (x_{ij})$
  contains no nonzero polynomials with fewer than $(r+1)!$ terms, and
  the only polynomials in $I$ with precisely $(r+1)!$ terms are those
  described in Theorem~\ref{thm:Mats}.
\end{cor}

\begin{proof}[Proof of the corollary]
Let $K$ be an algebraic closure of $L$. Then any polynomial $f$ in $I$
vanishes on all matrices in $K^{m \times n}$ of rank at most $r$. Hence $f$
is of the form in Theorem~\ref{thm:Mats}.
\end{proof}

\begin{re}
  We do not know whether Corollary~\ref{cor:Mats} still holds if one
  allows to do an arbitrary invertible linear change of the $n^2$
  coordinates.  We suspect that this cannot reduce the minimal number
  of monomials in a nonzero polynomial in the ideal $I$.
\end{re}

\begin{proof}[Proof of Theorem~\ref{thm:Mats}]\ \\
\noindent {\bf Part 1: proof of the lower bound $(r+1)!$} We proceed
by induction on~$r$. For $r=0$ the statement is evidently true. Now we
suppose that $r \geq 1$ and that the statement is true for $r-1$.

Let $f$ be a nonzero polynomial that vanishes on all rank-$r$
matrices.  Then $m,n>r$.  Furthermore, since the matrices of rank at
most $r$ form an affine cone, any homogeneous component of $f$ also
vanishes on them, hence we may assume that $f$ is homogeneous of
positive degree.

Let $x_m=(x_{m1},\ldots,x_{mn})$ be variables representing the last
row of the matrix, and write
\[ f=\sum_{\alpha \in S} f_\alpha x_m^\alpha. \]
where $S$ is a finite subset of $\ZZ_{\geq 0}^n$ and the
$f_\alpha$ are nonzero polynomials in the entries of the
first $m-1$ rows. 
If some $f_\alpha$ vanishes identically on rank-$r$
matrices, then we replace
$m$ by $m-1$ and $f$ by that~$f_\alpha$.  After finitely many such
steps, we reach a situation where no $f_\alpha$ vanishes identically on
rank-$r$ matrices.

Each $f_\alpha$ vanishes on every rank-$(r-1)$ matrix of size
$(m-1) \times n$. Indeed, if $A$ is such a matrix, then $f(A,x_m)$ is
the zero polynomial because appending any $m$-th row to~$A$ yields a
matrix of rank at most $r$, on which $f$ was assumed to vanish.  By
the induction assumption, each $f_\alpha$ has at least $(r-1)!$ terms.

On the other hand, since no $f_{\alpha}$ vanishes on all rank-$r$
matrices, for any very general $(m-1) \times n$-matrix $A$ of
rank~$r$, we have $f_\alpha(A) \neq 0$ for all $\alpha \in S$.  Now
$f(A,x_m)$ vanishes identically on the $r$-dimensional row space
of~$A$. We may further assume that the row space $U \subseteq K^n$ of
$A$ is very general in the sense of
Proposition~\ref{prop:Linear}. Then, by that proposition, $f(A,x_m)$
has at least $r+1$ terms, and hence $f$ has at least
$(r+1) \cdot r!=(r+1)!$ terms.  \medskip

\noindent
{\bf Part 2: proof of the characterisation.} 
Now assume that equality holds. Then by Proposition~\ref{prop:Linear},
$f(A,x_m)$ is a monomial times a linear combination of $p^a$-th powers
of variables, for some $a \in \ZZ_{\geq 0}$. After dividing by that
monomial, it is just a linear combination of $p^a$-th powers of
variables. Furthermore, the same argument applies to {\em any} row or
column of the matrix, so (after discarding rows and columns on which
$f$ does not depend, and dividing by suitable monomials) $f$ is a
linear combination of $p^a$-th powers of the variables in {\em every}
row/column and involves precisely $r+1$ of them. {\em A priori}, the
exponents $p^a$ depend on the row/column, though if the entry on
position $(i,j)$ appears in $f$ then the exponent $p^a$ for the $i$-th
row and that for the $j$-th column are the same.

This leads us to consider a bipartite graph $\Gamma$ on $[m] \sqcup [n]$ with an
edge $(i,j)$ if the variable $x_{i,j}$ appears in $f$. The
graph $\Gamma$ is 
regular of degree $r+1$, and this implies that $n=m$.  If $x_{i,j}$
appears in $f$, then it does so with exponent $p^a$, and we give
the edge $(i,j)$ the label $a$. The edge labels are constant on connected
components of $\Gamma$. Let $M_i \sqcup N_i,\ i=1,\ldots,q$ be the
vertex sets of those connected components. So both the $M_i$ and the
$N_i$ form partitions of $[m]=[n]$; the $M_i$ label rows, and the $N_i$
label columns. Regularity of the graph implies that $|M_i|=|N_i|$.
After reordering row indices and column indices, we may assume that
$M_1,M_2,M_3,\ldots,M_q$ are consecutive intervals, and that $N_i=M_i$.
Then $f$ depends only on the variables in the blocks of a block diagonal
matrix with square diagonal blocks labelled by $M_1 \times N_1, M_2 \times N_2,
\ldots,M_q \times N_q$.

Let $a_i$ be the common edge label of the edges between the edges in
$M_i$ and $N_i$, i.e.\ all variables $x_{kl}$ with $k \in M_i$ and $l \in
N_i$ appear with exponent $p^{a_i}$ in $f$. By basic linear
algebra (Lemma~\ref{lm:Completion} below)
any $q$-tuple of diagonal blocks $A_i \in K^{M_i \times
N_i}$ for $i=1,\ldots,q$ that are all of rank $\leq r$ can be extended
to a matrix $A \in K^{m \times n}$ of rank at most $r$, and hence $f$
vanishes on such a tuple $(A_1,\ldots,A_q)$.  Now applying the field
automorphism $\alpha \mapsto \alpha^{p^{-a_i}}$ to all entries in $A_i$
yields a matrix $\tilde{A}_i$ which is again of rank $r$, and hence $f$
vanishes on the $q$-tuple $(\tilde{A}_1,\ldots,\tilde{A}_q)$. But this
means that the polynomial $\tilde{f}$ obtained from $f$ by replacing each
$x_{kl}^{p^{a_i}}$ (with $(k,l) \in M_i \times N_i$) by $x_{kl}$ vanishes
on $(A_1,\ldots,A_q)$. By construction, $\tilde{f}$ vanishes
on all rank-$r$ matrices, has $(r+1)!$ terms
and is now multilinear in the $m$ rows and $m$ columns. We are done
if we can show that $q=1$,
$M_1=N_1=[r+1]$, and $\tilde{f}$ is a scalar multiple of the $(r+1)
\times (r+1)$-determinant.

Without loss of generality we have
\[ \tilde{f}=\tilde{f}_1 x_{m,1} + \cdots + \tilde{f}_{r+1} x_{m,r+1}, \]
where $\tilde{f}_j$ is a polynomial with $r!$ terms that is
multilinear in the first $m-1$ rows and in all but the
$j$-th column and that vanishes on all rank-$(r-1)$
matrices. 

We again proceed by induction on $r$. First consider the base case
where $r=2$. By Remark~\ref{re:Binomials}, each $\tilde{f}_j$ is of
the form a constant times $x^\alpha-x^\beta$ where
$\alpha,\beta \in \ZZ^{[m-1] \times ([m] \setminus \{j\})}$ are
permutation matrices. Thus $\tilde{f}$ itself has $6$ monomials of the
form $x^\gamma$, where $\gamma \in \ZZ^{[m] \times [m]}$ is a
permutation matrix, and these terms have the property that for each
$\gamma$ there is precisely one $\gamma' \neq \gamma$ whose last row
agrees with that of $\gamma$. This argument applies to all
rows. Furthermore, $\tilde{f}$ vanishes on all rank-$2$ matrices,
hence in particular on the matrix $(a_i+b_j)_{i,j}$ where $a$ and $b$
are vectors of variables. Evaluating $x^\gamma$ on this matrix yields
\[ \sum_{I \subseteq [m]} \left(\prod_{i \in I} a_i \right) \cdot
  \left(\prod_{i \in [m] \setminus I} b_{\gamma(i)}\right) \]
where we have abused notation and written $\gamma \in S_m$ for the
permutation determined by $\gamma(i)=j$ if and only if the permutation
matrix $\gamma$ has a $1$ on position $(i,j)$.  Then the monomial
corresponding to $I$ uniquely determines and is determined
by~$\gamma(I)$.  In $\tilde{f}$, this monomial appears with a nonzero
coefficient in the term corresponding to $\gamma$, so it appears in at
least one more term. By Lemma~\ref{lm:S3} below, we have $m=3$, and
$\tilde{f}$ is a scalar multiple of the $3 \times 3$-determinant.

If $r \geq 3$, then, by induction, each $\tilde{f}_j$ is a one-term
multiple of an $r$-minor in the
$[m-1] \times ([m] \setminus \{j\})$-submatrix, and a similar
expansion exists for all rows and columns. Then Lemma~\ref{lm:Dets}
below shows that $m=r$ and that $f$ is a scalar multiple of the
$(r+1)$-minor, as desired.
\end{proof}

\begin{lm} \label{lm:S3}
Let $n$ be a natural number, $S_n$ the symmetric group, and $P \subseteq
S_n$ a subset with $|P|=6$ such that for all $I,J \subseteq [n]$,
the set $\{\pi \in P \mid \pi(I)=J\}$ has cardinality $0$ or $\geq 2$,
and cardinality equal to $0$ or $2$ if $|I|=|J|=1$. Then $n=3$ and
$P=S_3$. 
\end{lm}

For the following proof we thank Rob Eggermont.
\begin{proof}
The assumptions on $P$ are preserved under left and right multiplication,
i.e., replacing $P$ by $\tau P \sigma^{-1}$ for any $\tau,\sigma \in
S_n$. Using left and right multiplication, we may assume that $P$ contains the
identity element $e$. Under this additional assumption on $P$ we may
not use left and right multiplication anymore, but we may still use
conjugation.  The set $P$ contains precisely one other element, which
we dub $\pi_{23}$, that maps $\{1\}$ to $\{1\}$, and after conjugating
we may assume that $\pi_{23}(2)=3$. 

The set $\{\pi \in P \mid \pi(\{1,2\})=\{1,2\}\}$ has 
cardinality at least $2$, contains $e$, and
hence contains at least one further element, which we dub $\pi_{12} \neq
e$. This does not map $\{1\}$ to $\{1\}$, and hence $\pi_{12}$
interchanges $1$ and $2$.  Similarly, $P$ contains an element
$\pi_{13}$ which interchanges $1$ and $3$. Furthermore, since
$\pi_{12}(2)=1$, $P$ contains a further element
$\pi_{132} \neq \pi_{12}$ that maps $\{2\}$ to $\{1\}$, and since
$\pi_{13}(3)=1$, $P$ contains one further element
$\pi_{123} \neq \pi_{13}$ that maps $\{3\}$ to $\{1\}$.

Now $\pi_{12},\pi_{13},\pi_{132},\pi_{123}$ do not map $\{2,3\}$ to
itself, but $e$ does, hence so does $\pi_{23}$.
The following summarises what we know about the permutations so far:
\begin{align*}
e&=\id_{[n]} & \pi_{23}&:1 \to 1, 2 \leftrightarrow 3 \\
\pi_{12}&:1 \leftrightarrow 2 &
\pi_{13}&:1 \leftrightarrow 3 \\
\pi_{132}&:2 \to 1 &
\pi_{123}&:3 \to 1
\end{align*}
and by construction all of these elements are distinct, so they
exhaust $P$. 

The set $\{1,k\}$ for $k > 3$ is mapped to itself by $e$,
hence there is at least one other element of $P$ that does so, and
$\pi_{12},\pi_{13},\pi_{132},\pi_{123}$ clearly do not, so $\pi_{23}(k)=k$
and $\pi_{23}$ is the transposition $(2,3)$. 

The set $\{3\}$ can only be mapped to $\{2\}$ by ($\pi_{23}$ and) $\pi_{132}$,
so we find that $\pi_{132}(3)=2$. Then, apart from $e$, $\pi_{12}$ is
the only element of $P$ that can map $\{3\}$ to itself, so it must do
so: $\pi_{12}(3)=3$. Now $\{3,k\}$ for $k>3$ is mapped to itself by $e$
and the only other element that can potentially do so is $\pi_{12}$,
so $\pi_{12}=(1,2)$. Using $\{2,k\}$ instead, we find that $\pi_{13}=(1,3)$.

Now $\pi_{12}$ maps $\{2,k\}$ for $k>3$ to $\{1,k\}$, and the only other
element that can do so is $\pi_{132}$, so we find that
$\pi_{132}=(1,3,2)$. Similarly, $\pi_{13}$ maps $\{3,k\}$ for $k>3$ to
$\{1,k\}$, and the only other element that can do so is $\pi_{123}$,
hence $\pi_{123}=(1,2,3)$. 

We have thus established that $P \subseteq S_3 \subseteq S_n$, but then $\pi(k)=k$ for all
$\pi \in P$ and $k>3$, and this violates the assumption in the lemma
that precisely zero or two permutations map $\{k\}$ to $\{k\}$. It
follows that $n=3$ and $P=S_3$.
\end{proof}

\begin{lm} \label{lm:Dets} Let $r \geq 3$ and let
  $f \in K[x_{ij} \mid i,j \in [m]]$ be a polynomial in the entries of
  a generic matrix $x = (x_{ij})$ with the following properties:
\begin{enumerate}
\item $f$ vanishes on all matrices of rank $r$; 

\item for every row index $i \in [m]$, $f$ admits an expansion
\[ f=x_{i,j_1} f_1 + \cdots + x_{i,j_{r+1}} f_{r+1} \]
where $j_1<\ldots<j_{r+1}$ and where each $f_l$ is a
polynomial in the entries of the $([m] \setminus \{i\})
\times ([m] \setminus \{j_l\})$-submatrix $z$ of $x$ of the
following form: a scalar times a monomial times some 
$r$-minor of $z$; 

\item and similarly for column indices.
\end{enumerate}
Then $m=r+1$ and $f$ is a scalar multiple of the $(r+1) \times
(r+1)$-determinant of $x$.
\end{lm}

In the proof of the lemma, we use that the ideal of all polynomials
vanishing on rank-$r$ matrices has a Gr\"obner basis consisting of
$(r+1)$-minors for the lexicographic order given by 
\[
x_{1,m}>x_{1,m-1}>\ldots>x_{1,1}>x_{2,m}>\ldots>x_{2,1}>x_{3,m}>\ldots>x_{m,1}.
\]
For these results see \cite[Theorem 1]{St89} and the overview article
\cite{BrunsConca03}.  Note that the leading term of an
$(r+1)\times (r+1)$-determinant equals the product of the entries of
the main anti-diagonal.

\begin{proof}
From the expansion, we see that each variable in $f$ is
contained in precisely $r!$ terms, and that each monomial in
$f$ is of the form $x^\gamma$ with $\gamma$ an $m \times m$-permutation
matrix. 

Next we count variables. Since $f$ contains $r+1$ variables in each row,
the total number of variables in $f$ equals $m(r+1)$. After
permuting the columns of $x$, we may
assume that the expansion along the first row looks as
follows: 
\[ f=x_{1,1} f_1 + \cdots + x_{1,r+1} f_{r+1}. \]
Each $f_l$ contains $r^2$ variables in its determinant, denoted
$\det_l$ in the following, plus $(m-1-r)$ further distinct variables
in a monomial~$u_l$.  The variables in $u_l$ are also distinct from
the variables in the $f_n$ with $n \neq l$, because the former already
appear in all $r!$ terms of $f_l$, and can thus not appear
again. Counting also the $r-1$ variables $x_{1,l}$, we see
$(r+1)(m-r)$ variables outside the $\det_{l}$. This means that the
determinants use only $(r+1)r$ variables.  But then they are the
$r$-minors of an $r \times (r+1)$- or $(r+1) \times r$-submatrix~$y$
of the last $m-1$ rows of~$x$.

For a contradiction assume $y$ is \emph{not} contained in the first
$r+1$ columns. Then we can permute the first $r+1$ columns of $x$ so
that $\det_{r+1}$ is not contained in the first $r+1$ columns and uses
$r$ consecutive columns with labels in $[m] \setminus \{r+1\}$ with at
least one label larger than $r+1$.  Then we can further arrange the
last $m-1$ rows of $x$ so that the the rows in $\det_{r+1}$ are
consecutive, and the variables in $u_{r+1}$ are arranged pointing in a
down-right direction as do the black squares in
Figure~\ref{fig:matrix}, with those in the first $r$ columns coming in
rows before those of $\det_{r+1}$, and those beyond the first $r+1$
columns coming in rows after $\det_{r+1}$.

\begin{figure}
\includegraphics{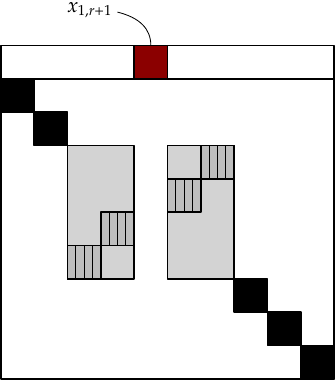}
\caption{The variables in the monomial $u_{r+1}$ are
represented by black boxes, the variables in the determinant in $f_{r+1}$
by gray boxes.}
\label{fig:matrix}
\end{figure}

Now consider the leading monomial of $f$ in the lexicographic order.
It is the product of the following factors: $x_{1,r+1}$, $u_{r+1}$
consisting of the black variables in Figure~\ref{fig:matrix}, and the
darker gray variables on the antidiagonal of the
$r \times r$-determinant in $f_{r+1}$. But this is not divisible by
the leading monomial of any $(r+1)$-minor, a contradiction showing
that $y$ is contained in the first $r+1$ columns of~$x$.
Then $y$ is in fact an $r \times (r+1)$ submatrix in the first $r+1$
columns of $x$; indeed, if it were an $(r+1) \times r$-submatrix, then
for any column index $j \in [r+1]$ appearing in $y$, $y$ could not
contain the $r$-minor $\det_j$ in $f_j$, simply because $y$ is too
narrow.

We relabel the rows such that $y$ is the submatrix of $x$ labelled by
$\{2,\ldots,r+1\} \times [r+1]$. Then each $f_j$ is the determinant
$\det_j$ of the
$\{2,\ldots,r+1\} \times ([r+1]\setminus \{j\})$-submatrix of $x$
times a constant $c_j$ times a monomial $u_j$ with a variable from
each of the last $m-r-1$ rows and the last $m-r-1$ columns. We claim
that all $u_j$ are equal. Indeed, let $g$ be $c_{r+1} u_{r+1}$ times
the $[r+1] \times [r+1]$-subdeterminant of $x$. Then
\begin{align*} 
g&=\sum_{j=1}^{r+1} (-1)^{j-1} c_{r+1} u_{r+1} x_{1,j} \det\ \hspace{-1ex}_j 
\text{ and }\\
h&:=f+(-1)^{r+1}g=\sum_{j=1}^{r+1} x_{1,j} (c_j u_j + 
(-1)^{r+j}c_{r+1} u_{r+1}) \det\ \hspace{-1ex}_j.
\end{align*}
In $h$ the term with $j=r+1$ cancels.  Now if $h$ is nonzero, then one
of the terms $(c_j u_j + (-1)^{r+j} c_{r+1} u_{r+1})\det_j$ with
$j<r+1$ is nonzero, and it does not vanish on any very general
rank-$r$ matrix $A \in K^{([m]\setminus \{1\}) \times [m]}$. But then,
as in the proof of Theorem~\ref{thm:Mats} above, $h(A,x_1)$ is a
nonzero polynomial with fewer than $r+1$ terms that vanishes
identically on the row space of $A$, a contradiction to
Proposition~\ref{prop:Linear}.

We conclude that $h=0$, and this implies that all $u_j$ are equal to
$u_{r+1}$. But then $f$ involves only one variable from each of the
last $m-(r+1)$ rows.  Since it also contains $r+1$ variables from each
of these, we conclude that $m=r+1$, and $f$ is a scalar
multiple of the determinant.
\end{proof}

We conclude this section with the following simple matrix
completion problem.

\begin{lm} \label{lm:Completion}
Let $m,m_1,m_2,r \geq 0$ be nonnegative integers, and suppose that
$m=m_1+m_2$. Denote by $X_m$ the variety of $m \times m$-matrices of
rank at most $r$. Then the projection $X_m \to X_{m_1} \times X_{m_2}$
that maps a matrix to its diagonal blocks is surjective. 
\end{lm}

In the proof of Theorem~\ref{thm:Mats} the corresponding
statement is used with $q$ factors, and this follows 
by induction from the case $q=2$. 

\begin{proof}
Let $(A_1,A_2) \in X_{m_1} \times X_{m_2}$. Then $A_i=B_i
\cdot C_i$ for certain $B_i \in K^{m_i \times r}, C_i \in K^{r
\times m_i}$. But then 
\[ X_m \ni \begin{bmatrix} B_1 \\ B_2 \end{bmatrix}
\cdot 
\begin{bmatrix} C_1 & C_2 \end{bmatrix}
=
\begin{bmatrix} A_1 & * \\ * & A_2 \end{bmatrix}  \]
is a matrix of rank at most $r$ with the desired diagonal
blocks. 
\end{proof}

\subsection*{Relations to polynomial identity testing}

After the first version of this paper was posted, Robert Andrews pointed
out to us that Theorem~\ref{thm:Mats} has a (modest) application to
polynomial identity testing for sparse polynomials. 
Consider the subset $P_{t,N}\subset\QQ[y_1,\ldots,y_N]$ of nonzero polynomials
with fewer than $t$ terms. For our restricted purpose, a {\em
  hitting set generator} for $P_{t,N}$ is a polynomial map
$\phi\colon \QQ^M\to \QQ^N$ such that $f \circ \phi$ is nonzero
for every $f \in P_{t,N}$. One typically wants $M$ to be much
smaller than $N$ and the components of $\phi$ to be easy-to-evaluate
polynomials in $M$ variables.

Assume that $t=(r+1)!$ and $N=n^2$.  By Theorem~\ref{thm:Mats}, the
multiplication map
\[ \QQ^{n \times r} \times \QQ^{r \times n} \to \QQ^{n
\times n} = \QQ^N,\ (A,B) \mapsto A\cdot B \]
is a degree-two hitting set generator of $P_{t,N}$. Expressing all in $t$
and $N$ and using Stirling's approximation, we have
\[ M \leq c \cdot N^{1/2} \cdot \log(t)/\log(\log(t)) \]
for some (explicit) constant $c$.  On the other hand, any degree-two
hitting set generator for $P_{t,N}$ with $t\le N$ necessarily has $M$
at least some constant times $N^{1/2}$, so the above is nearly
optimal.
It should be mentioned, though, that polynomial identity testing for
polynomials with a bounded number of terms can be done
deterministically in polynomial time \cite{KlivansSpielman01}, so that
this hitting set generator may not be immensely useful.

For another link of our work to polynomial identity testing, we refer
to \cite{AndrewsForbes21}, where it is shown that any nonzero element
$f$ in the ideal generated by $(r+1) \times (r+1)$-minors can be used
as an oracle in the construction of a small circuit that approximately
computes the $s \times s$-determinant, for $s=\Theta(r^{1/3})$. This can
be understood as expressing that such a polynomial has high {\em border
complexity}, a different measure of complexity than the number of terms
considered in this paper.

\section{No short polynomials vanish on bounded-rank 
    skew symmetric matrices}
\label{sec:skew-symmetric}

We now focus on square and skew-symmetric matrices $A$; it is well known
that these have even rank. The coordinates on the space of
skew-symmetric $n \times n$-matrices are, say, the
$\binom{n}{2}$ matrix entries strictly below the diagonal. 

Let $r$ be an even integer. If $A$ has rank at most $r$, then in
particular all principal $(r+2)$-Pfaffians vanish on $A$.  These
Pfaffians have $(r+1)!!=(r+1) \cdot (r-1) \cdots \cdot 1$ terms, in
bijection with the perfect matchings in the complete graph on $r+2$
vertices. This is fewer than the $(r+1)!$ from the previous section,
except when $r=0$, when the two agree.  The following theorem says
that there are no shorter polynomials.

\begin{thm}\label{thm:skew}
  Let $r$ be even and let $m \geq r$.  There is no nonzero polynomial
  vanishishing on all skew-symmetric $m \times m$-matrices of rank
  $\leq r$ that has fewer than $(r+1)!!$ terms.  Furthermore, any
  polynomial with $(r+1)!!$ terms that vanishes on all skew-symmetric
  $m \times m$-matrices of rank $r$ is a one-term multiple of a
  $p^e$-th power of some principal $(r+2)$-Pfaffian, for some
  $e \in \ZZ_{\geq 0}$.
\end{thm}

Before proceeding with the proof, we record an immediate consequence of
the theorem.

\begin{cor} \label{cor:skew} Let $r$ be even and let $m \geq r$. For
  any field $L$, the ideal $I$ in the polynomial ring
  $K[x_{ij} \mid 1 \leq i<j \leq m]$ generated by the $r$-Pfaffians of
  the matrix $x$ does not contain polynomials with fewer than
  $(r+1)!!$ terms, and the only polynomials in $I$ with $(r+1)!!$
  terms are those in Theorem~\ref{thm:skew}.
\end{cor}

\begin{proof}
Such a polynomial vanishes on all skew-symmetric matrices in
$K^{m \times m}$, where $K$ is an algebraic closure of $L$.
Now apply Theorem~\ref{thm:skew}. 
\end{proof}

\begin{proof}[Proof of Theorem~\ref{thm:skew}]
  The proof proceeds along the same lines as that of
  Theorem~\ref{thm:Mats}. Again, we proceed by induction on $r$. For
  $r=0$, the $2$-Pfaffians are precisely the matrix entries, which of
  course are the shortest nonzero polynomials vanishing on the zero
  matrix.  \medskip

\noindent
{\bf Part 1: proof of the lower bound $(r+1)!!$}
Assume that $r \geq 2$ and decompose 
\[ f=\sum_{\alpha \in S} f_\alpha x_m^\alpha. \]
where $x_m$ consists of the first $m-1$ entries of the last row---which are,
up to a sign, also the first $m-1$ entries of the last
column---and where the $f_\alpha$ are nonzero polynomials in the
(lower-triangular) entries of the top left $(m-1) \times (m-1)$-block. 

Now all $f_\alpha$ vanish on all skew symmetric matrices $A$ of rank at most
$r-2$. Indeed, for an arbitrary row vector $u \in K^{m-1}$,
the skew symmetric matrix
\[ \begin{bmatrix} A & -u^T \\ u & 0 \end{bmatrix} \]
has rank at most $r$, and hence $f$ vanishes on it.  Therefore, if
some $f_\alpha(A) \neq 0$, then $f(A,x_m)$ is a nonzero polynomial
that vanishes identically on $K^{m-1}$, a contradiction since $K$ is
infinite. From the induction hypothesis, we conclude that each
$f_\alpha$ has at least $(r-1)!!$ terms, with equality if and only if
it is a one-term multiple of a $p^e$-th power of some principal
$r$-Pfaffian.

We may further assume that no $f_\alpha$ vanishes identically on
rank-$r$ skew-symmetric matrices; otherwise, we would replace $f$ by
$f_\alpha$. Pick a very general skew symmetric matrix
$A \in K^{(m-1) \times (m-1)}$ of rank $r$. Then
$f_{\alpha}(A) \neq 0$ for all $\alpha$, and we claim that $f(A,x_m)$
is a polynomial that vanishes identically on the row space of
$A$. Indeed, if $u$ is in the row space of $A$, then appending it to
$A$ as an $m$-th row does not increase the rank of $A$, and then
appending $-u^T$, along with a zero, as the last column, could only
increase the rank by $1$, but since a skew-symmetric matrix has even
rank, it does not.  Hence $f$ vanishes on the resulting matrix, and
thus $f(A,x_m)$ vanishes on the row space of $A$.  By
Proposition~\ref{prop:Linear}, at least $r+1$ of the $f_\alpha$ are
nonzero. Therefore, $f$ has at least $(r+1) \cdot ((r-1)!!)=(r+1)!!$
terms, as desired.  \medskip

\noindent
{\bf Part 2: proof of the characterisation.}
Assume that equality holds. By Proposition~\ref{prop:Linear}, after
dividing $f$ by a monomial in the variables of the last row,
discarding rows (and corresponding columns) on which $f$ does not
depend, and rearranging columns if necessary, the $x_m^\alpha$ are
equal to $x_{m,i}^{p^a}$ for some common exponent~$a$.  The same
applies to all rows.  Like in the case of ordinary matrices, we
construct an undirected graph $\Gamma$, now not necessarily bipartite,
on $[m]$ in which $\{i,j\}$ is an edge if and only if $x_{ij}$ appears
in~$f$. The exponents $a$ are constant on the connected components of
$\Gamma$, and by the same argument as in the proof of
Theorem~\ref{thm:Mats}, now using Lemma~\ref{lm:Completion2} below for
the matrix completion, we may replace $f$ by an $\tilde{f}$ which is
linear in the rows. By Lemma~\ref{lm:Pfaffs} below, $m=r+2$ and
$\tilde{f}$ is a scalar multiple of a Pfaffian; in particular,
$\Gamma$ is connected and $f$ is a $p^a$-th power of $\tilde{f}$.
\end{proof}

\begin{lm} \label{lm:Pfaffs}
Let $r \geq 2$ be even. Assume that $f$ is a polynomial in the entries
of a generic skew-symmetric matrix $x=(x_{ij})_{ij}=(-x_{ji})_{ij}$
with the following properties:
\begin{enumerate}
\item $f$ vanishes on all skew matrices of rank $r$; and
\item for every row index $i$, $f$ admits an expansion
\[ f=x_{i,j_1} f_1 + \cdots + x_{i,j_{r+1}} f_{r+1} \]
where $j_1<\ldots<j_{r+1}$ are all distinct from $i$ and where each $f_l$
is a polynomial in the entries of the 
$([m] \setminus \{i,j_l\})^2$-submatrix $z$ of $x$ with the
following shape: a scalar times a monomial times the
Pfaffian of a principal $r \times r$-submatrix of $z$. 
\end{enumerate}
Then $m=r+2$ and $f$ itself is a scalar multiple of the Pfaffian
of $x$.
\end{lm}

In the proof of this lemma we use that the ideal of polynomials
vanishing on rank-$r$ skew-symmetric matrices is generated by the
$(r+2)$-Pfaffians of principal submatrices, and that these form a
Gr\"obner basis with respect to the lexicographic order with
\[
x_{1,n}>x_{1,n-1}>\ldots>x_{1,2}>x_{2,n}>\ldots>x_{2,3}>\ldots>x_{n-1,n}. 
\] 
The leading term of the Pfaffian of the principal matrix with row
indices $j_1<\ldots<j_{r+2}$ of $x$ is
$x_{j_1,j_{r+2}} x_{j_2,j_{r+1}} \cdots
x_{j_{(r+2)/2},j_{1+(r+2)/2}}$, a product of $(r+2)/2$ variables in
the upper half of the matrix pointing in the down-left direction. For
these results see \cite{HerzogTrung92}.

\begin{proof}
  First we count variables: $f$ contains precisely $r+1$ variables
  from each row, but every variable appears in two rows, so $f$
  contains $m(r+1)/2$ variables in total. Thus $m$ is even.

On the other hand, consider the expansion along the first row: 
\[
  f=x_{1,j_1} f_1 + \cdots + x_{1,j_{r+1}} f_{r+1}.
\]
Here each $f_l$ is a scalar times a monomial $u_l$ times an
$r$-Pfaffian, denoted~$\Pf_l$. The Pfaffian contains $\binom{r}{2}$
variables, and $u_l$ another $(m-(r+2))/2$ variables, disjoint from
those in $\Pf_l$---here we use that $f$ is linear in the variables in
each row. Furthermore, the variables in $u_l$ appear in all $r!!$
terms of $f_l$, and hence, since all variables in $f$ appear in
precisely that many terms, the variables in $u_l$ are disjoint from
the variables in the $f_{l'}$ with $l' \neq l$.

Hence in total we see $(r+1)(m-(r+2))/2$ distinct variables in
$u_1,\ldots,u_{r+1}$.  Adding to these the $r+1$ variables
$x_{1,j_l}$, there are only $\binom{r+1}{2}$ variables left for the
$r+1$ Pfaffians~$\Pf_l$, $l=1,\ldots,r+1$. This is only possible if
those Pfaffians are the sub-Pfaffians of a principal
$(r+1) \times (r+1)$-submatrix $y$ of the
$([m]\setminus \{1\})^2$-submatrix of~$x$.

Let $J \subseteq [m] \setminus \{1\}$ be the set of indices labelling
the columns (and rows) of $y$. We have $|J|=r+1$, and claim that
$J=\{j_1,\ldots,j_{r+1}\}$.  Suppose not, and then let $\Pf_i$ be the
Pfaffian of a matrix involving a column index
$j \in J \setminus \{j_1,\ldots,j_{r+1}\}$. After applying a
permutation of $[m] \setminus \{1\}$ to rows and columns, we may
assume that $i=r+1$ and that $j=m>j_{r+1}=m-1$. After applying a
further permutation of $\{2,\ldots,j_{r+1}-1\}=\{2,\ldots,m-2\}$, we
may assume that $\Pf_{r+1}$ is the Pfaffian of the principal submatrix
with columns $m-r,m-r+1,\ldots,m-2,m$. The variables in $u_{r+1}$
encode a partition of $\{2,\ldots,m-r-1\}$ into pairs. After applying
a permutation of this set to rows and columns, we may assume that
these pairs are $\{2,3\},\{4,5\},\ldots,\{m-r-2,m-r-1\}$.  See
Figure~\ref{fig:matrix3} for an illustration. Now the leading monomial
of $f$ equals $x_{1,j_{r+1}}=x_{1,m-1}$ times $u_{r+1}$ times the
leading monomial of $\Pf_{r+1}$; the latter is indicated by dark gray
squares in Figure~\ref{fig:matrix3}. But this monomial contains no
$(r+2)/2$ variables arranged in a down-left direction, hence $f$ does
not lie in the Pfaffian ideal, a contradiction, showing that
$J = \{j_1,\ldots,j_{r+1}\}$.

\begin{figure}
\includegraphics{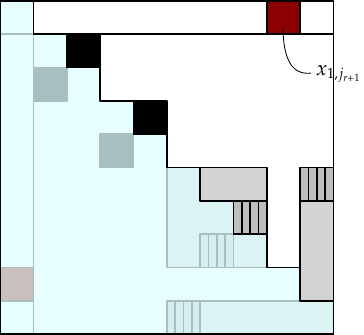}
\caption{The lower half of the matrix is drawn for
visualisation purposes only. The black squares correspond to
variables in $u_{r+1}$ and the gray region indicates the
Pfaffian $\Pf_{r+1}$ in $f_{r+1}$.}
\label{fig:matrix3}
\end{figure}

After all $y$ has rows and columns by $j_1,\ldots,j_{r+1}$.  Applying
a permutation of $[m] \setminus \{1\}$ to rows and columns, we may
assume that $j_1=2,j_2=3,\ldots,j_{r+1}=r+2$.  Then each $f_l$ equals
a scalar times $x_{1,l+1}$ times the Pfaffian $\Pf_l$ of the
$([m] \setminus \{1,l+1\})^2$-submatrix of~$x$, times a monomial
$u_{l}$ whose variables live in the last $m-(r+2)$ rows and columns
of~$x$. Now let $g$ be the unique scalar multiple of $u_{r+1}$ times
the $(r+2)$-Pfaffian in the upper left corner of $x$ such that in
$h:=f-g$ the terms involving $x_{1,r+2}$ cancel.  As in the proof of
Theorem~\ref{thm:skew} above, if $h$ is nonzero, then for a
very general skew-symmetric $([m] \setminus \{1\})^2$-matrix
$A$ of rank $r$, $h(x_1,A)$, where $x_1$ stands for the variables in
the first row of $x$, is a linear polynomial with fewer than $r+1$
terms that vanishes on the row space of $A$. Again, this contradicts
Proposition~\ref{prop:Linear}.

Hence $h=0$ and $f$ equals a scalar multiple of $u_{r+1}$ times a
Pfaffian. But since $f$ contains $r+1$ variables from all of the last
$m-(r+2)$ columns, we find that $m=r+2$ and $f$ is a scalar
multiple of a Pfaffian, as desired.
\end{proof}

\begin{lm} \label{lm:Completion2}
Let $m,m_1,m_2,r \geq 0$ be nonnegative integers with $r$ even, and
suppose that $m=m_1+m_2$. Denote by $X_m$ the variety of skew-symmetric
$m \times m$-matrices of rank at most $r$. Then the projection $X_m
\to X_{m_1} \times X_{m_2}$ that maps a matrix to its diagonal blocks
is surjective.
\end{lm}

\begin{proof}
Write $r=2s$ and $(A_1,A_2) \in X_{m_1} \times X_{m_2}$. Then $A_i=B_i
\cdot C_i-C_i^T \cdot B_i^T$ for certain $B_i \in K^{m_i \times s}, C_i \in K^{s
\times m_i}$. But then 
\[ X_m \ni \begin{bmatrix} B_1 \\ B_2 \end{bmatrix}
\cdot 
\begin{bmatrix} C_1 & C_2 \end{bmatrix}
- \begin{bmatrix} C_1^T \\ C_2^T \end{bmatrix}
\cdot 
\begin{bmatrix} B_1^T & B_2^T \end{bmatrix}
=
\begin{bmatrix} A_1 & * \\ * & A_2 \end{bmatrix}  \]
is a skew-symmetric matrix of rank at most $r$ with the desired diagonal
blocks. 
\end{proof}

\section{Symmetric matrices}
\label{sec:symmetric-matrices}

An $(r+1)$-minor $\det x[I,J]$ of a symmetric matrix of variables can
have various numbers of terms, depending on $|I \cap J|$: if
$I \cap J=\emptyset$, then this determinant has $(r+1)!$ terms, while
for the other extreme, where $I=J$, the number of terms equals the
number of collections of necklaces that can be made with $n$ distinct
beads; much less than $(r+1)!$. These counts assume that
$\cha K \neq 2$ since the coefficients in the determinant for $I=J$
are (plus or minus) powers of~$2$.

We guess that, if $\cha K \neq 2$, then in the ideal generated by all
$(r+1)$-minors, the shortest polynomials are those
of the form $\det(x[I,I])$ with $I \subseteq [n]$ of size $r+1$. But to
prove this, one would like to perform a Laplace expansion like was used
in the proofs of Theorems~\ref{thm:Mats} and~\ref{thm:skew}.  Such a
Laplace expansion, in the symmetric case, naturally involves determinants
of matrices $x[I',J']$ with $I' \neq J'$, and so to prove our guess one
would probably need to work with a stronger induction hypothesis. At
present, we do not know how to approach this challenge.

\bibliographystyle{amsplain}
\bibliography{shortpolynomials}

\bigskip \medskip

\noindent
\footnotesize {\bf Authors' addresses:}

\smallskip

\noindent Jan Draisma, Universität Bern, Switzerland; and
Eindhoven University of Technology, the Netherlands, 
{\tt jan.draisma@math.unibe.ch}

\noindent Thomas Kahle, OvGU Magdeburg, Germany,
{\tt thomas.kahle@ovgu.de}

\noindent Finn Wiersig, University of Oxford, UK,
{\tt finn.wiersig@maths.ox.ac.uk}

\end{document}